\documentclass[11pt]{article}

\usepackage{amsmath}
\usepackage{amssymb}
\usepackage{amsthm}
\usepackage{url}
\usepackage{latexsym}
\usepackage{amsbsy}
\usepackage{color}
\usepackage[colorlinks=true, linkcolor=blue, filecolor=brown,citecolor=blue,pagebackref=true,driverfallback=true]{hyperref}

\newcommand{\forme}[1]{}

\newcommand{\set}[1]{\left\{#1\right\}}

\newcommand{\To}{\longrightarrow}

\newcommand{\R}{\mathcal{R}}

\newcommand{\C}{\mathcal{C}}

\newcommand{\T}{\mathcal{T}}

\newtheorem{theorem}{Theorem}[section]
\newtheorem{corollary}[theorem]{Corollary}
\newtheorem{lemma}[theorem]{Lemma}
\newtheorem{proposition}[theorem]{Proposition}

\newtheorem{prf}{Proof}

\theoremstyle{definition}
\newtheorem{definition}{Definition}[section]
\newtheorem{exmpl}[definition]{Example}
\newtheorem{remark}[theorem]{Remark}

\def\thm{\begin{theorem}}
\def\ethm{\end{theorem}}
\def\prop{\begin{proposition}}
\def\eprop{\end{proposition}}
\def\rem{\begin{remark}}
\def\erem{\end{remark}}
\def\defi{\begin{definition}}
\def\edefi{\end{definition}}
\def\nmu{\begin{enumerate}}
\def\enmu{\end{enumerate}}
\def\qtn{\begin{equation}}
\def\eqtn{\end{equation}}
\def\qtnl{\begin{equation*}}
\def\eqtnl{\end{equation*}}
\def\lem{\begin{lemma}}
\def\elem{\end{lemma}}
\def\cor{\begin{corollary}}
\def\ecor{\end{corollary}}
\def\prf{\begin{proof}}
\def\eprf{\end{proof}}
\def\css{\begin{cases}}
\def\ecss{\end{cases}}
\def\exm{\begin{exmpl}}
\def\eexm{\end{exmpl}}
\title{\textbf{Coherent configurations over copies of association schemes of prime order}}
\author{Reza Sharafdini \footnote{Corresponding author}
\\
{\small Department of Mathematics, Persian Gulf University}\\
{\small Bushehr 7516913817, Iran}\\
{\small sharafdini@pgu.ac.ir}\\
Mitsugu Hirasaka\\
{\small Department of Mathematics, College of Natural Sciences, Pusan National University,}\\
{\small Busan 609--735, South Korea.}\\
{\small hirasaka@pusan.ac.ir}
}

\begin{document}

\maketitle

\begin{abstract}
Let $G$ be a group acting faithfully and transitively on $\Omega_i$ for $i=1,2$.
A famous theorem by Burnside implies the following fact:
If $|\Omega_1|=|\Omega_2|$ is a prime and the rank of one of the actions is greater than two,
then the actions are equivalent, or equivalently $|(\alpha,\beta)^G|=|\Omega_1|=|\Omega_2|$
for some $(\alpha,\beta)\in \Omega_1\times \Omega_2$.

In this paper we consider a combinatorial analogue to this fact
through the theory of coherent configurations, and give
some arithmetic sufficient conditions for a coherent configuration with two homogeneous components of prime order to be
uniquely determined by one of the homogeneous components.
\end{abstract}

\section{Introduction}
A famous theorem by Burnside states that each transitive permutation group of prime degree
with rank greater than two is Frobenius or regular. Since any Frobenius group of prime degree is a subgroup of one-dimensional affine group, it follows that such a permutation group is uniquely determined by its rank and degree up to equivalence of group actions.
Especially, if a group acts faithfully, transitively but not 2-transitively
on each of two sets of the same prime size,
then the two actions are equivalent. Let us formulate this fact in the following two paragraphs.

Let $G$ be a group acting transitively on $\Omega_i$ for $i=1,2$.
Then $G$ acts on $\Omega_i\times \Omega_j$
by
\[
(\alpha,\beta)^g=(\alpha^g,\beta^g)\quad\quad\mbox{for $(\alpha,\beta)\in \Omega_i\times \Omega_j$ \quad and \quad $g\in G$},
\]
for all $i,j=1,2$.
It is well-known that (e.g., see \cite[Lemma~1.6B]{dixon}) the following are equivalent:
\begin{enumerate}
\item[(a)] The action of $G$ on $\Omega_1$ is
equivalent to that on $\Omega_2$;
\item[(b)]  There exists $(\alpha,\beta)\in \Omega_1\times \Omega_2$
such that $G_\alpha=G_\beta$;
\item[(c)] There exists $(\alpha,\beta)\in \Omega_1\times \Omega_2$
such that $|(\alpha,\beta)^G|=|\Omega_1|=|\Omega_2|$.
\end{enumerate}
Note that the rank of the action of $G$ on $\Omega_i$ is equal to the number of orbits of $G$ on $\Omega_i\times \Omega_i$,
and if $G$ acts faithfully on $\Omega_i$,
then $G$ can be identified with a permutation group of $\Omega_i$.

Suppose that $G$ acts faithfully on $\Omega_i$ with $i=1,2$ and $|\Omega_1|=|\Omega_2|$ is a prime.
Then, as mentioned in the first paragraph, these actions are equivalent if the rank of one of the actions is greater than two,
and so there exists an orbit $R$ of $G$ on $\Omega_1\times \Omega_2$
such that $|R|=|\Omega_1|=|\Omega_2|$.

In this paper we consider a combinatorial analogy to this fact
through the theory of coherent configurations. The concept of coherent configurations was first introduced by Higman who published a series of papers (e.g., \cite{Higman1975}, \cite{Higman1976}, \cite{Higman1987})
to associate a lot of important criterions with group actions.

Here we define a coherent configuration, its intersection numbers and its fibers
according to the notations as in \cite{Inp2009}.
\begin{definition}\label{def:1}
Let $V$ be a finite set and $\mathcal{R}$ a partition of $V\times V$. We say that the pair $\mathcal{C}=(V,\mathcal{R})$ is a \textit{coherent configuration} if it satisfies the following:
\begin{enumerate}
\item The diagonal relation $\Delta_V$ is a union of elements of $\mathcal{R}$
where we denote $\{(u,u)\mid u\in U\}$ by $\Delta_U$ for a set $U$.
\item For each $R\in \mathcal{R}$ its transpose $R^t=\set{(u,v)\mid (v,u)\in R}$ is an element of $\mathcal{R}$.
\item For all $R,S,T\in \mathcal{R}$ there exists a constant $c_{RS}^T$ such that\\
$$\mbox{$c_{RS}^T=|R(u)\cap S^t(v)|$ for all $(u,v)\in T$},$$
where we denote by $T(w)$ the set $\{z\in V\mid (w,z)\in T\}$ for $w\in V$ and $T\in \mathcal{R}$.
\end{enumerate}
\end{definition}
The constants $c_{RS}^T$ are called the \textit{intersection numbers}.
A subset $X$ of $V$ is called a \textit{fiber} of $\mathcal{C}$ if $\Delta_X\in \mathcal{R}$.
We denote the set of all fibers of $\mathcal{C}$ by $\mathrm{Fib}(\mathcal{C})$.
By Definition~\ref{def:1}(i), $V$ is partitioned into the fibers of $\mathcal{C}$, and by Definition~\ref{def:1}(i),(iii),
$\mathcal{R}$ is partitioned into
\[\Big\{\mathcal{R}_{X,Y}\mid X,Y\in \mathrm{Fib}(\mathcal{C})\Big\}\:\:\mbox{where}\:\: \mathcal{R}_{X,Y}=\Big\{R\in \mathcal{R}\mid R\subseteq X\times Y\Big\}.\]
Let $U$ be a union of fibers of $\mathcal{C}$.
Then the pair
\[\Big(U,\{R\in\mathcal{R}\mid R\subseteq U\times U\}\Big),\]
is also a coherent configuration, which is denoted by $\mathcal{C}_U$.

For $R\in \mathcal{R}_{X,Y}$ we denote $c_{RR^t}^{\Delta_X}$ by $d_R$.
Then, by two-way counting we have
\begin{equation}\label{eq:dc1}
|R|=d_R|X|=d_{R^t}|Y|.
\end{equation}

For $X\in \mathrm{Fib}(\mathcal{C})$, $\mathcal{C}_X$ is nothing but an \textit{association scheme}, i.e.,
a coherent configuration with only one fiber (see \cite{bi} or \cite{Zieschang2005} for its background).
For short we shall write $\mathcal{R}_{X,X}$ as $\mathcal{R}_X$ and
$\C_X$ is called a \textit{homogeneous component} of $\mathcal{C}$.

A general question here is formulated as follows: what can be said about the coherent configuration if its homogeneous components are known. 
For example, it is a well-known fact that a coherent configuration $\C$ corresponds to a system of linked block designs if $|\R_{X}|=2$ for all $X\in \mathrm{Fib}(\mathcal{C})$.
After the seminal Hanaki-Uno theorem on association schemes of prime order (see \cite{Han_Uno2006} or Theorem \ref{thm:hu}), it seems quite
natural to ask on a possible structure of a coherent configuration each homogeneous component of
which is of prime order.
The following is our first main result answering to this question:
\begin{theorem}\label{thm:main1}
Let $X,Y\in \mathrm{Fib}(\mathcal{C})$ such that $|X|=|Y|$ is a prime.
Then $|\mathcal{R}_{X,Y}|\in \{1, |\mathcal{R}_X|\}$.
In particular, if $|\mathcal{R}_{X,Y}|>1$, then
\[|\mathcal{R}_{X,Y}|=|\mathcal{R}_X|=|\mathcal{R}_Y|.\]
\end{theorem}

In order to state our second main theorem we need to recall the following observation.
Let $G$ be a group acting on a finite set $\Omega$.
Then $G$ acts on $\Omega\times \Omega$ componentwise, and
an orbit of $G$ on $\Omega\times \Omega$ is called an \textit{orbital} (or 2-orbit) of $G$.
We denote the set of orbitals of $G$ by $\mathcal{O}_G$.
Then it is well-known that $\mathcal{C}_G=(\Omega,\mathcal{O}_G)$ is a coherent configuration,
and $\mathrm{Fib}(\mathcal{C}_G)$ is the set of orbits of $G$ on $\Omega$.
In this sense, a coherent configuration is a combinatorial object to generalize
the orbitals of a group action.

Now we assume that $\mathcal{C}=(V,\mathcal{R})$ is a coherent configuration with exactly two fibers $X$, $Y$.
Then (\ref{eq:dc1}) proves the equivalence of the first two statements of the following (see \cite{HR} for the remaining):
\begin{enumerate}
\item[(d)] There exists $R\in \mathcal{R}_{X,Y}$ such that $|R|=|X|=|Y|$.
\item[(e)] $1\in \{d_R\mid R\in \mathcal{R}_{X,Y}\}\cap \{d_R\mid R\in \mathcal{R}_{Y,X}\}$.
\item[(f)] $\mathcal{C}$ is isomorphic to $\mathcal{C}_X\bigotimes \mathcal{T}_2$
where $\T_n=\Big(\{1,2,\ldots, n\},\,\big\{\{(i,j)\}\mid 1\leq i,j\leq n\big\}\Big)$
(see Section~2 for the definition of isomorphism and $\bigotimes$).
\end{enumerate}
We notice the following:\\
(d) is a combinatorial analogy to (c), and such $R$ is a matching between $X$ and $Y$;
(e) is a simple arithmetic condition on intersection numbers;
(f) implies that $\mathcal{C}_X$ and $\mathcal{C}_Y$ are isomorphic, and $\mathcal{C}$ is uniquely determined by $\mathcal{C}_X$.

In this paper we aim to obtain the analogous conclusion (d)--(f) to (a)--(c).
The following is our second main result
to generalize the fact as in the first paragraph under
certain arithmetic conditions on intersection numbers:

\begin{theorem}\label{thm:main2}
Suppose that $\mathcal{C}=(V,\mathcal{R})$ is a coherent configuration with exactly two fibers $X$, $Y$ satisfying
\begin{equation}\label{eq:003}
\mbox{$|X|=|Y|$ is a prime, $|\mathcal{R}_{X,X}|>2$ and $|\mathcal{R}_{X,Y}|>1$.}
\end{equation}
Then there exists $R\in \mathcal{R}_{X,Y}$ such that $|R|=|X|=|Y|$
if one of the following conditions holds with $k=\dfrac{|X|-1}{|\mathcal{R}_{X,X}|-1}$:
\begin{enumerate}
\item[{\rm(i)}] $|\mathcal{R}_{X,X}|>k^2(k+e-2)$ where $e$ is the number of prime divisors of $k$;
\item[{\rm(ii)}] $k\in\{q,2q,3q\}$ for some prime power $q$;
\item[{\rm(iii)}] $k=4q$ for some prime power $q$ with $3\nmid q+1$.
\end{enumerate}
\end{theorem}

Let us show the reason why we exclude the case of $|\mathcal{R}_{X,X}|=2$.
Each symmetric design induces the coherent configuration with exactly two fibers
and eight relations (see \cite{Higman-SRD95} or \cite[Example~1.3]{HR}), and if the design is a non-trivial
one on a prime number of points, like the Fano plane,
then the induced coherent configuration does not satisfy (d)--(f).

Of course, if $|\mathcal{R}_{X,Y}|=1$, then none of (d)--(f) hold,
while $\mathcal{C}$ is the direct sum of $\mathcal{C}_X$ and $\mathcal{C}_Y$
(see \cite{HR} for the definition of direct sum).


\begin{remark}
Applying Theorem~\ref{thm:main2} for $\mathcal{C}_{X\cup Y}$ with $|X|<100$
we obtain the same conclusion as Theorem \ref{thm:main2} except for the case $(|X|,k)=(71,35)$ (see Section \ref{sec:ad} for the details).
\end{remark}

Suppose that
\begin{equation}\label{eq11}
\mbox{$(|X|,k)=(71,35)$ and $1\notin \{d_R\mid R\in \mathcal{R}_{X,Y}\}$.}
\end{equation}
Then by Theorem \ref{thm:main1}, $|\mathcal{R}_{X,Y}|=3$. The three elements of $\mathcal{R}_{X,Y}$ must
form three symmetric designs whose parameters $(v,k,\lambda)$
are $(71,35,17)$, $(71,21,6)$ and $(71,15,3)$, respectively.
Though each of such symmetric designs exists (see \cite{Ademaj-71-21-6}, \cite{Beker-Haemers71-15-3}, \cite{Crnk71-35-17}, \cite{Haemers71-15-3} and \cite{Janko-71-21-6} or \cite[II.6.24,VI.16.30]{Han-Com-De}),
it does not guarantee the existence of a coherent configuration satisfying (\ref{eq11}).

In \cite{Higman-SRD95}, Higman gave a result to eliminate the case of $(|X|,k)=(71,35)$
as in the previous paragraph.
But, the proof given in \cite[(3.2)]{Higman-SRD95} contains a serious gap, so the result may not be recognized to be true,
while we have not found any counterexample. We would be able to disprove \cite[(3.2)]{Higman-SRD95}
if there exists a coherent configuration satisfying (\ref{eq11}).

In Section~\ref{sec:pre} we prepare several basic results on intersection numbers
and introduce the concepts of complex products and equitable partitions.
In Section~\ref{sec:proof1} we give a proof of Theorem~\ref{thm:main1}.
In Section~\ref{sec:proof2}
we give a proof of Theorem~\ref{thm:main2}.
We add Section~\ref{sec:ad} for the elimination of coherent configurations on at most 200 points satisfying (\ref{eq:003}).

\section{Preliminaries}\label{sec:pre}
Throughout this section we assume that $\mathcal{C}=(V,\mathcal{R})$ is a coherent configuration.

Let $\mathcal{C}_i=(V_i,\mathcal{R}_i)$  be a coherent configurations, $i=1,2$.

An \textit{isomorphism} from $\C_1$ to $\C_2$ is defined to be
a bijection $\psi:V_1\cup\R_1\To V_2\cup\R_2$ such that for all $u,v\in V_1$ and $R\in\R_1$,
\[(u,v)\in R\Longleftrightarrow\big(\psi(u),\psi(v)\big)\in\psi(R).\]
We say that $\C_1$ is \textit{   isomorphic} to $\C_2$ and denote it by $\C_1\simeq \C_2$ if there exists an isomorphism from $\C_1$ to $\C_2$.

We set
\[\R_1\otimes\R_2=\Big\{R_1\otimes R_2\mid R_1\in\R_1,\ R_2\in\R_2\Big\},\]
where
\[R_1\otimes R_2=\Big\{\big((u_1,u_2),(v_1,v_2)\big)\mid (u_1,v_1)\in
R_1,\ (u_2,v_2)\in R_2\Big\}.\]
Then $\big(V_1\times V_2,\R_1\otimes\R_2\big)$ is a coherent configuration called the {\it tensor product} of $\C_1$ and $\C_2$ and
denoted by $\C_1\bigotimes\C_2$.

Following \cite{Zieschang2005} we define the \textit{complex product} on the power set of $\mathcal{R}$.
For all subsets $\mathcal{S}$ and $\mathcal{T}$ of $\mathcal{R}$
we define the complex product $\mathcal{S}\mathcal{T}$ of $\mathcal{S}$ and $\mathcal{T}$
to be the subset
\[\Big\{R\in \mathcal{R}\mid\exists (S,T)\in \mathcal{S}\times\mathcal{T}; c_{ST}^R>0 \Big\}.\]
The complex product is an associative binary operation on the power set of $\mathcal{R}$
where the proof is parallel to that for association schemes (see \cite{Zieschang2005}).
For convenience we shall write $\mathcal{S}\{T\}$, $\{S\}\mathcal{T}$ and $\{S\}\{T\}$  as
$\mathcal{S}T$, $S\mathcal{T}$ and $ST$, respectively.

In this paper we need intersection numbers $c_{RS}^T$ for
$R\in \mathcal{R}_{X,Y}$, $S\in \mathcal{R}_{Y,Z}$
and $T\in \mathcal{R}_{X,Z}$ under the assumption $|X|=|Y|=|Z|$.
The following is a collection of simplified equations on such intersection numbers
(see \cite{mp} or \cite[Lemma~2.2]{HR} for general formed equations
\footnote{
We missed to assume that
all fibers of $\mathcal{C}$ have the same size at Lemma~2.2 in \cite{HR}
where the lemma is used only for such coherent configurations in \cite{HR}. }).
For $\mathcal{U}\subseteq \mathcal{R}$ we shall write $d_{\mathcal{U}}$ instead of
$\sum_{U\in \mathcal{U}}d_U$.

\begin{lemma}\label{lem:int}
For all $X$, $Y$, $Z\in \mathrm{Fib}(\mathcal{C})$ with $|X|=|Y|=|Z|$ and
all $R\in \mathcal{R}_{X,Y}$, $S\in \mathcal{R}_{Y,Z}$
and $T\in \mathcal{R}_{X,Z}$
we have the following:
\begin{enumerate}
\item $d_Rd_S=\displaystyle\sum_{T\in\R_{X,Z}}c_{RS}^{T}d_T$;
\item $\mbox{$c_{RS}^Td_T=c_{TS^t}^Rd_R=c_{R^t T}^Sd_S$~and~${\rm lcm}(d_R,d_S)\mid c_{RS}^Td_T$}$;
\item $\mbox{ $|\{U\in \mathcal{R}\mid c_{RS}^U>0\}|\leq {\rm gcd}(d_R,d_S)$, i.e., $|RS|\leq {\rm gcd}(d_R,d_S)$;}$
\item $|X|=d_{\mathcal{R}_{X,X}}=d_{\mathcal{R}_{X,Y}}$.
\end{enumerate}
\end{lemma}

The following lemmata were proved in \cite[Lemma~2.3, Lemma~2.2]{pc}{\footnote{Though it is a statement for association schemes,
a parallel way to the proof can be applied for balanced coherent configurations. }:

\begin{lemma}\label{lem:crst}
For all $S,T\in \mathcal{R}_{X,Y}$ with $|X|=|Y|$, we have
\[\mbox{
$SS^t\cap TT^t\subseteq \{\Delta_X\}$ if and only if $ c_{S^t T}^R\leq 1$ for each $R\in \mathcal{R}$.} \]
\end{lemma}

\begin{lemma}\label{lem:psd}
Let $Z\in \mathrm{Fib}(\mathcal{C})$ such that $|Z|$ is a prime. Then for  each $R\in \mathcal{R}_Z\setminus\{\Delta_Z\}$ we have:
\begin{enumerate}
\item  $d_R=k$
where $k=\dfrac{|Z|-1}{|\mathcal{R}_Z|-1}$;
\item $\sum_{S\in \mathcal{R}_Z}c_{SS^t}^R=k-1$.
\end{enumerate}
\end{lemma}


According to \cite{martin} or \cite{hkk} we define an equitable partition of a homogeneous component.
\begin{definition}
Let $X\in \mathrm{Fib}(\mathcal{C})$ and $\Pi=\{C_1,C_2,\ldots, C_m\}$ be
a partition of $X$, i.e.,
\[X=\bigcup_{i=1}^m C_i,\:\: \mbox{$C_i\cap C_j\neq\emptyset$ if $i\ne  j$, and $C_i \neq \emptyset$ for each $i=1,2,\ldots, m$.}\]
An element of $\Pi$ is called a \textit{cell}.
We say that $\Pi$ is an \textit{equitable partition} of $\mathcal{C}_X$
if, for all $i,j=1,2,\ldots, m$ and each $R\in \mathcal{R}_X$, $|R(x)\cap C_j|$ is constant whenever $x\in C_i$.
\end{definition}
For example, $\{X\}$ and $\{\{x\}\mid x\in X\}$ are equitable partitions of $\mathcal{C}_X$.

For each $Y\in \mathrm{Fib}(\mathcal{C})$ and each $y\in Y$ we define
\begin{equation}\label{eq:piy}
  \Pi_y:=\big\{T(y)\mid T\in \mathcal{R}_{Y,X}\big\}.
\end{equation}

Then $\Pi_y$ is an equitable partition of $\mathcal{C}_X$, since
\[\mbox{$|R(x)\cap S(y)|=c_{RS^t}^T$ whenever $x\in T(y)$.}\]

\section{Proof of Theorem \ref{thm:main1}}\label{sec:proof1}

In \cite{Han_Uno2006} Hanaki and Uno proved the following brilliant theorem:
\begin{theorem}\label{thm:hu}
All non-principal irreducible characters of an association scheme of
prime order are algebraic conjugate and of degree one.
\end{theorem}
The following proposition is obtained as a consequence of the previous theorem:
\begin{proposition}\label{prop:alg}
Let $\mathcal{C}=(V,\mathcal{R})$ be an association scheme of prime order
and $\Pi$ be an equitable partition of $\mathcal{C}$.
Then $|\Pi|\equiv 1\mod{|\mathcal{R}|-1}$.
\end{proposition}

\begin{proof}
Let $\mathcal{A}$ denote the adjacency algebra of $\mathcal{C}$ over $\mathbb{C}$.
Then the subspace $W$ spanned by the characteristic vectors of the cells in $\Pi$ is a left $\mathcal{A}$-module
with respect to the ordinary matrix product.
Since $\mathcal{A}$ is semi-simple, $W$ is a direct sum of irreducible submodules.

Note that the subspace spanned by the all-one vector is an $\mathcal{A}$-submodule of $W$
affording the principal character, and its multiplicity is one.

Since the character afforded by $W$ is integral valued, it is left invariant from
any algebraic conjugate action.
It follows from Theorem~\ref{thm:hu} that
all non-principal irreducible submodules of $W$ have the same multiplicity, say $m$.
Since
\[\mbox{$\mathrm{dim}_{\mathbb{C}}(W)=|\Pi|$ and $\mathrm{dim}_{\mathbb{C}}(\mathcal{A})=|\mathcal{R}|$,}\]
it follows that
\[ |\Pi|=1+m(|\mathcal{R}|-1).\]
\end{proof}

\begin{proof}[\textbf{Proof of Theorem~\ref{thm:main1}}]
Let $\mathcal{C}=(V,\mathcal{R})$ be a coherent configuration with
$X$, $Y\in \mathrm{Fib}(\mathcal{C})$ such that $|X|=|Y|$ is a prime. Recall that $\Pi_y$ is an equitable partition of $\mathcal{C}_X$ where $y\in Y$.
By \eqref{eq:piy}, $|\Pi_y|=|\mathcal{R}_{X,Y}|$. Then it follows from Proposition~\ref{prop:alg}
that
\[|\mathcal{R}_{X,Y}|\equiv 1\mod{|\mathcal{R}_X|-1}.\]
Since $|\mathcal{R}_{X,Y}|\leq |\mathcal{R}_X|$
(see \cite[p.223]{Higman1987} or \cite[Proposition~2.7]{HR}), $|\mathcal{R}_{X,Y}|\in \{1, |\mathcal{R}_X|\}$. Applying the first statement for $\mathcal{C}_Y$ with $|\mathcal{R}_{X,Y}|\leq |\mathcal{R}_Y|$, we obtain the second statement.

\end{proof}

\section{Proof of Theorem \ref{thm:main2}}\label{sec:proof2}

For the remainder of this paper we assume that
$\mathcal{C}=(V,\mathcal{R})$ is a coherent configuration with
$X$, $Y\in \mathrm{Fib}(\mathcal{C})$ such that
\[\mbox{$m=|X|=|Y|$ is a prime, $r=|\mathcal{R}_{X}|>2$ and $|\mathcal{R}_{X,Y}|>1$.}\]
By Theorem \ref{thm:main1}, we have
\[r=|\mathcal{R}_X|=|\mathcal{R}_{X,Y}|=|\mathcal{R}_Y|.\]
For the remainder of this paper we set
\[k=\frac{m-1}{r-1}.\]
By Lemma~\ref{lem:psd}(i)   the multi-set $(d_R\mid R\in \mathcal{R}_{Z})$ with $Z\in \{X,Y\}$
coincides with $(1,k,\ldots, k)$ by a suitable ordering. In this section we aim to show that $1\in\{d_R\mid R\in \mathcal{R}_{X,Y}\}$, which implies that
the multi-set $(d_R\mid R\in \mathcal{R}_{X,Y})$ coincides with $(1,k,\ldots, k)$ by
a suitable ordering, since the complex product $SR$ is a singleton with $d_{SR}=d_S$ whenever
$S\in \mathcal{R}_X$ and $d_R=1$ by Lemma~\ref{lem:int}(iii).

\begin{lemma}\label{lem:ba}
For all $S,T\in \mathcal{R}_{X,Y}$ with $S\ne T$ we have the following:
\begin{enumerate}
\item[{\rm(i)}] $d_Sd_S\equiv d_S\mod{k}$;
\item[{\rm(ii)}] $d_Sd_T\equiv 0\mod{k}$.
\end{enumerate}
\end{lemma}
\begin{proof}
(i) Applying Lemma~\ref{lem:int}(i) for $S$ and $S^t$ with $d_S=d_{S^t}$ and $c_{SS^t}^{\Delta_X}=d_S$, we obtain that
\[
d_Sd_S=d_S+k\sum_{\substack{T\in \mathcal{R}_{X,X}\\T\neq \Delta_X}}c_{SS^t}^{T}.
\]

(ii) Applying Lemma~\ref{lem:int}(i) for $S$ and $T^t$ with $d_T=d_{T^t}$ and $\Delta_X\notin ST^t$, we obtain that
\[
d_Sd_T=k\sum_{T\in \mathcal{R}_{X,X}}c_{ST^t}^{T}.
\]
\end{proof}

We set
\[\mathcal{S}_1:=\{T\in\mathcal{R}_{X,Y}\mid k\nmid d_T\},\:\:\:\
\mathcal{S}_2:=\{T\in\mathcal{R}_{X,Y}\mid d_T=k\}\:\:\mbox{and}\]
\[\mathcal{S}_3:=\{T\in\mathcal{R}_{X,Y}\mid k\mid d_T,\:\: k<d_T\}.\]

\begin{lemma}\label{lem:bound}
Let $k=p_1^{\alpha_1}\cdots p_e^{\alpha_e}$ where $p_i$ are the distinct prime divisors of $k$ and
$\alpha_i$ are positive integers.
Then we have the following:
\begin{enumerate}
\item For each $i=1,\ldots, e$ there exists a unique $S\in \mathcal{R}_{X,Y}$ such that
$p_i\nmid d_S$;
\item $|\mathcal{S}_1|\leq e$;

\item $k|\mathcal{S}_3|+d_{\mathcal{S}_1}\leq 1+k(e-1)$.
\end{enumerate}
\end{lemma}
\begin{proof}
(i) By Lemma~\ref{lem:int}(iv) and Lemma~\ref{lem:psd}(i),
\[m=1+(r-1)k\equiv 1 \mod{p_i}.\]
Since $m=d_{\mathcal{R}_{X,Y}}$, there exists an $S\in \mathcal{R}_{X,Y}$ such that
$p_i\nmid d_S$. The uniqueness of such $S$ is a direct consequence of Lemma~\ref{lem:ba}(ii).

(ii) The correspondence given in (i) gives a function from $\{p_1,p_2,\ldots, p_e\}$ to $\mathcal{S}_1$.
It remains to show that this function is onto.

Let $S\in\mathcal{S}_1$. By the definition of $\mathcal{S}_1$, there exists $p_i$ such that ${p_i}^{\alpha_i}$ does not divide $d_S$.
By Lemma~\ref{lem:ba}(i),
\[d_Sd_S\equiv d_S \mod{k}.\]
Therefore $d_S(d_S-1)$ is divided by $k$. Since $d_S$ and $d_S-1$ are relatively prime, $p_i^{\alpha_i}\nmid d_S$ implies that $p_i\nmid d_S$.
It follows from (i) that $d_S$ lies in the range of the function.

(iii) Note that $r=|\mathcal{S}_1|+|\mathcal{S}_2|+|\mathcal{S}_3|$ and
\[ m=\sum_{S\in \mathcal{R}_{X,Y}}d_S
=\sum_{i=1}^3d_{\mathcal{S}_i}\geq  d_{\mathcal{S}_1}+k|\mathcal{S}_2|+2k|\mathcal{S}_3|.\]
Since $k|\mathcal{S}_2|+k|\mathcal{S}_3|=k(r-|\mathcal{S}_1|)$ and $m=1+k(r-1)$,
it follows that
\[1+k(|\mathcal{S}_1|-1)\geq d_{\mathcal{S}_1}+k|\mathcal{S}_3|.\]
By (ii), we have
\[1+k(e-1)\geq  d_{\mathcal{S}_1}+k|\mathcal{S}_3|.\]
This completes the proof of (iii).
\end{proof}

\begin{lemma}\label{lem:min}
We have $\max\{d_S\mid S\in \mathcal{R}_{X,Y}\}\leq k \cdot \min\{d_S\mid S\in \mathcal{R}_{X,Y}\}$.
\end{lemma}
\begin{proof}
Let $S,T\in \mathcal{R}_{X,Y}$ such that
\[\mbox{$d_S=\min\{d_S\mid S\in \mathcal{R}_{X,Y}\}$ and
$d_T:=\max\{d_S\mid S\in \mathcal{R}_{X,Y}\}$.}\]
Then $T\in RS$ for some $R\in \mathcal{R}_X$ since $T\in \mathcal{R}_{X}S$.
Applying Lemma~\ref{lem:int}(i) we have $d_T\leq kd_S$.
\end{proof}

For $S\in\mathcal{R}_{X,Y}$ we define
\[\mathcal{U}_S:=\big\{R\in \mathcal{R}_{X}\mid R^t R\cap SS^t=\{\Delta_X\}\big\}.\]

\begin{lemma}\label{lem:key}
For each $S\in \mathcal{R}_{X,Y}$ we have the following:
\begin{enumerate}
\item $r-|\mathcal{U}_S|\leq (d_S-1)(k-1)$.

\item If $R\in \mathcal{U}_S- \{\Delta_X\}$,
then $k$ divides $d_T$ for each $T\in RS$.

\item If $\mathcal{U}_SS\cap \mathcal{S}_2= \emptyset$,
then $r< d_S(k+e-2)$.
\end{enumerate}
\end{lemma}
\begin{proof}
(i) Note that
\[\mathcal{R}_X-\mathcal{U}_S=\bigcup_{R_1\in SS^t-\{\Delta_X\}}\{R\in \mathcal{R}_X\mid R_1\in R^t R\}.\]
By Lemma~\ref{lem:int}(iii) with $c_{SS^t}^{\Delta_X}>0$,
\[|SS^t-\{\Delta_X\}|\leq d_S-1.\]
It follows from Lemma~\ref{lem:psd}(ii) that
\[|\{R\in \mathcal{R}_X\mid R_1\in R^t R\}|\leq\sum_{R\in \mathcal{R}}c_{R^t R}^{R_1}=k-1.\]
This implies that
\[r-|\mathcal{U}_S|=|\mathcal{R}_X-\mathcal{U}_S|\leq (d_S-1)(k-1).\]

(ii) It is an immediate consequence of Lemma~\ref{lem:int}(ii) and Lemma~\ref{lem:crst}.

(iii) Suppose that
\[\mathcal{U}_SS\cap \mathcal{S}_2= \emptyset.\]
Then we have
\[\mathcal{U}_SS\subseteq \mathcal{R}_{X,Y}-\mathcal{S}_2.\]
It follows from (ii) that
 \[(\mathcal{U}_S-\{\Delta_X\})S\subseteq \mathcal{S}_3.\]
By Lemma~\ref{lem:bound}(iii) and Lemma~\ref{lem:min},
 \begin{equation}\label{eq1}
 d_{\mathcal{S}_3}\leq d_Sk |\mathcal{S}_3|\leq d_S[1+k(e-1)-d_{ \mathcal{S}_1 }].
 \end{equation}
On the other hand, applying Lemma~\ref{lem:crst} and Lemma~\ref{lem:int}(iv) for the first inequality and (i) for the second one,
\begin{equation}\label{eq2}
 d_{\mathcal{U}_S S}\geq 1+(|\mathcal{U}_S|-1)k-d_S\geq 1+[r-(d_S-1)(k-1)-1]k.
\end{equation}
Since $(\mathcal{U}_S-\{\Delta_X\})S\subseteq \mathcal{S}_3$,
\[d_{\mathcal{U}_S S}-d_S\leq d_{(\mathcal{U}_S-\{\Delta_X\})S}\leq d_{\mathcal{S}_3}.\]
It follows from (\ref{eq1}) and (\ref{eq2})  that
 \[ 1+[r-(d_S-1)(k-1)-1]k-d_S \leq d_S[1+k(e-1)-d_{\mathcal{S}_1}],\]
and hence,
\[r\leq \frac{d_S}{k}[2+k(e-1)-d_{\mathcal{S}_1}]-\frac{1}{k}+(d_S-1)(k-1)+1.\]
Thus,
\[r\leq d_S[\frac{2}{k}+e-1-\frac{ d_{\mathcal{S}_1}}{k}+k-1]-k+2-\frac{1}{k}<d_S(k+e-2).  \]
This completes the proof of (iii).
\end{proof}

\begin{proposition}\label{prop:1}
If $r>k^2(k+e-2)$
where $e$ is the number of prime divisors of $k$,
then $1\in \{d_S\mid S\in \mathcal{R}_{X,Y}\}$.
\end{proposition}
\begin{proof}
We claim that
\[\min\{d_S\mid S\in \mathcal{R}_{X,Y}\}\leq k.\]
If not, then
\[1+k(r-1)=m=\sum_{S\in \mathcal{R}_{X,Y}}d_S> kr,\]
a contradiction.

By Lemma~\ref{lem:min},
\[\max\{d_S\mid S\in \mathcal{R}_{X,Y}\}\leq k^2.\]
Applying the contraposition of Lemma~\ref{lem:key}(iii) we have
\[\mbox{$\mathcal{U}_SS\cap \mathcal{S}_2\ne \emptyset$ for each $S\in \mathcal{R}_{X,Y}$,}\]
and hence,
$T\in RS$ for some $R\in \mathcal{U}_S$ and $T\in \mathcal{S}_2$.
Since $d_T=k$ and $c_{RS}^T=1$ by Lemma~\ref{lem:crst},
$d_S$ divides $k$ for each $S\in \mathcal{R}_{X,Y}$.
This implies that $|\mathcal{S}_3|=0$.

We claim $|\mathcal{S}_1|=1$.
Suppose not.
Since $1+(r-1)k=m=d_{\mathcal{S}_1}+k(r-|\mathcal{S}_1|)$,
\[1+k|\mathcal{S}_1|\leq k+\sum_{S\in \mathcal{S}_1}d_S\leq k+k/2+k/2+(|\mathcal{S}_1|-2)k,\]
a contradiction.

By the claim we have $\mathcal{S}_1=\{S\}$ for some $S\in \mathcal{R}_{X,Y}$.
Since
\[1+k(r-1)=m=k|\mathcal{S}_2|+d_S=k(r-1)+d_S,\]
we have $d_S=1$.
This completes the proof.
\end{proof}

\begin{lemma}\label{lem:rst}
If $S,T\in \mathcal{R}_{X,Y}$ with $ST^t=\{R\}$,
then \[\mbox{$c_{RR^t}^{R_1}\geq d_T$ for each $R_1\in SS^t$ and $c_{R^t R}^{R_2}\geq d_S$ for each $R_2\in TT^t$.}\]
\end{lemma}
\begin{proof}
Let $y\in Y$, $x_1,x_2\in S^t(y)$ and $z\in T^t(y)$.
Note that $(x_i,z)\in R$ for $i=1,2$ since $S T^t =\{R\}$.
Since $z\in T^t(y)$ is arbitrarily taken,
we have $T^t(y)\subseteq R(x_1)\cap R(x_2)$, which proves
the first statement. By the symmetric argument the second statement can be proved.
\end{proof}

\begin{proposition}\label{prop:lambda}
There exist no $S,T\in \mathcal{R}_{X,Y}$ such that
\begin{equation}\label{eq3}
\mbox{$ST^t=\{R\}$, $d_S+d_T\geq k+1$ and $1<d_S<d_T$.}
\end{equation}
\end{proposition}
\begin{proof}
Suppose that $S,T\in \mathcal{R}_{X,Y}$ satisfies (\ref{eq3}).

We claim that $SS^t=\{\Delta_X,R_1\}$ for some $R_1\in \mathcal{R}_X-\{\Delta_X\}$.
Suppose not, i.e., $SS^t-\{\Delta_X\}$ has at least two elements $R_1$, $R_2$.
By Lemma~\ref{lem:int}(i),
\[k^2=d_Rd_{R^t}\geq  k+c_{RR^t}^{R_1}d_{R_1}+c_{RR^t}^{R_2}d_{R_2}= k+c_{RR^t}^{R_1}k+c_{RR^t}^{R_2}k.\]
It follows from Lemma~\ref{lem:rst} and $d_S+d_T\geq k+1$ that
\[k^2\geq k(k+2),\]
a contradiction.

We claim that $SS^t\cap TT^t=\{\Delta_X,R_1\}$.
Suppose not, i.e., $SS^t\cap TT^t=\{\Delta_X\}$.
Then, by Lemma~\ref{lem:crst}, $c_{ST^t}^R=1$.
It follows from Lemma~\ref{lem:int}(i) that $k=d_R=d_Sd_T$,
which contradicts
$d_S+d_T\geq k+1$ and $1<d_S<d_T$.

We claim that $R=R^t$.
Suppose not, i.e., $R\ne R^t$.
Then, by Lemma~\ref{lem:psd}(ii),
\[k-1=\sum_{R_2\in \mathcal{R}_X}c_{R_2 R_2^t}^{R_1}\geq c_{RR^t}^{R_1}+c_{R^t R}^{R_1}\geq d_S+d_T\geq k+1,\]
a contradiction.

We claim that $TT^t=\{\Delta_X,R_1\}$.
If $R_2\in TT^t-\{\Delta_X,R_1\}$,
then $c_{R R}^{R_2}\geq d_S$ by Lemma~\ref{lem:rst} with $R=R^t$.
By Lemma~\ref{lem:int}(i),
\[k^2=d_{R}d_{R}\geq k+c_{RR}^{R_1}k+c_{RR}^{R_2}k,\]
which implies that $k\geq 1+d_T+d_S$, a contradiction to $d_S+d_T\geq k+1$.

We claim that $c_{R_1R_1^t}^{R_1}\geq d_T-2$.
By the previous claim, for all $z_1,z_2\in T^t(y)$ with $z_1\ne z_2$
we have $(z_1,z_2)\in R_1$.
Thus, \[c_{R_1R_1^t}^{R_1}=|R_1(z_1)\cap R_1(z_2)|\geq |T^t(y)-\{z_1,z_2\}|\geq d_T-2.\]

Since $c_{R_1R_1^t}^{R_1}+c_{RR^t}^{R_1}\geq d_T-2+d_T\geq k$ by Lemma~\ref{lem:rst},
it follows from Lemma~\ref{lem:psd}(ii)
that $R=R_1$. Thus, $c_{RR^t}^R=k-1$
since $1<d_S$ and
\[\mbox{$S^t(y)\cup T^t(y)\setminus \{x_1,x_2\}\subseteq
R(x_1)\cap R(x_2)$ for $x_1,x_2\in S^t(y)$.}\]

Since $\{\Delta_X,R\}$ is closed under the complex product, $1+k$ divides $|X|$.
Since $|X|$ is a prime, it follows that $\{\Delta_X,R\}=\mathcal{R}_{X}$, and hence $|\mathcal{R}_X|=2$,
a contradiction.
\end{proof}

\begin{lemma}\label{lem:41}
Suppose that $k=4q$ for some prime power $q$ and $1\notin \{d_S\mid S\in \mathcal{R}_{X,Y}\}$.
Then $|\mathcal{S}_3|=0$, $|\mathcal{S}_1|=2$, and $\{d_S\mid S\in \mathcal{S}_1\}=\{3q,q+1\}$.
\end{lemma}
\begin{proof}
By Lemma~\ref{lem:bound}(iii) and the assumption, $|\mathcal{S}_3|=0$.
By Lemma~\ref{lem:bound}(ii), $|\mathcal{S}_1|\leq 2$.
Let $S\in \mathcal{S}_1$.
Then, by Lemma~\ref{lem:ba}, $d_S\equiv 1\mod{q}$.
By the assumption, $1<d_S<4q$.
Since $d_S\leq d_{\mathcal{S}_1}\leq 1+4q$ Lemma~\ref{lem:bound}(iii),
it follows from Lemma \ref{lem:ba} that
\[d_S\in \{q+1,3q+1\}.\]
Let $T\in \mathcal{R}_{X,Y}$ with $S\ne T$.
Since $d_Sd_T\equiv 0\mod{4q}$ by Lemma~\ref{lem:ba},
$q\mid d_T$.
Since $m=1+k(r-1)=d_{\mathcal{S}_1}+d_{\mathcal{S}_2}=d_S+d_T+k(r-2)$,
we have $d_S+d_T=k+1$.
Therefore, we conclude from Proposition~\ref{prop:lambda} that $\{d_S\mid S\in \mathcal{S}_1\}=\{3q,q+1\}$.
\end{proof}

\begin{proof}[\textbf{Proof of Theorem~\ref{thm:main2}}]
(i) is a direct consequence of Proposition~\ref{prop:1}.

(ii) Suppose on the contrary that \[1\notin \{d_S\mid S\in \mathcal{R}_{X,Y}\}.\]
Note that $e\leq 2$ if $k\in \{q,2q,3q\}$ for some prime power $q$.
By Lemma~\ref{lem:bound}(iii), $|\mathcal{S}_3|=0$, and $d_{\mathcal{S}_1}\leq k+1$.
Since
\[1+k(r-1)=d_{\mathcal{S}_1}+d_{\mathcal{S}_2}\leq k+1+d_{\mathcal{S}_2},\]
we have $d_{\mathcal{S}_2}\geq k(r-2)$, and, hence,
$|\mathcal{S}_2|\geq r-2$.

Suppose $k=q$. Then the statement follows from Lemma~\ref{lem:bound}(iii) since $e=1$.

Suppose $k=2q$. Then $|\mathcal{S}_1|\leq 2$ and
$\{d_S\mid S\in \mathcal{S}_1\}=\{q,q+1\}$ by Lemma~\ref{lem:bound}(ii),(iii) and Lemma~\ref{lem:ba}.
Without loss of generality we assume that
\[\mbox{$\mathcal{S}_1=\{S,T\}$, $d_S=q+1$ and $d_T=q$.}\]
Since $q$ and $q+1$ are relatively prime, it follows from Lemma~\ref{lem:int}(iii) that
$ST^t=\{R\}$ for some $R\in\mathcal{R}$, which contradicts Proposition~\ref{prop:lambda}.

Suppose $k=3q$.
Then we have either
\[\mbox{$\{d_S\mid S\in \mathcal{S}_1\}=\{q,2q+1\}$
or $\{d_S\mid S\in \mathcal{S}_1\}=\{2q,q+1\}$.}\]

The first case is done by Proposition~\ref{prop:lambda}.

For the last case we assume that
$\mathcal{S}_1=\{S,T\}$, $d_S=q+1$ and $d_T=2q$.
By Lemma~\ref{lem:int}(i),(ii),
$SS^t=\{\Delta_X,R\}$ for some $R\in \mathcal{R}$ with $R=R^t$.
This implies that $k=d_R$ is even since $|X|$ is an odd prime, so $q$ is a power of two.
Thus, $d_S$ and $d_T$ are relatively prime.
Therefore, the statement follows from Lemma~\ref{lem:int}(iii) and Proposition~\ref{prop:lambda}.

(iii) Suppose $k=4q$. Then, by Lemma~\ref{lem:41},
$\{d_S\mid S\in\mathcal{R}_{X,Y}\}=\{q,3q+1\}$ or $\{d_S\mid S\in\mathcal{R}_{X,Y}\}=\{3q,q+1\}$.
The statement follows from the assumption and Proposition~\ref{prop:lambda}.
\end{proof}

\section{Appendix}\label{sec:ad}
In this section we show how Theorem~\ref{thm:main2} is applied to small configurations $\mathcal{C}_{X\cup Y}$ with $|X|=|Y|<100$.

First, we denote by $\mathcal{M}$ the set of primes $m$ less than 100.

Second, we take the set $\mathcal{K}$ of positive integers $k$ such that
\[\mbox{$k\mid m-1$ for some $m\in \mathcal{M}$ with $k<m-1$ and}\]
\[k\notin \{q,2q,3q \mid \mbox{$q$ is a prime power} \} \cup \{ 4q\mid \mbox{ $q$ is a prime power with $3\nmid q+1$} \}.\]
Then $\mathcal{K}=\{20,30,35,44\}$.

\begin{lemma}\label{lem44}
If $k=20$, then $1\in \{d_S\mid S\in \mathcal{R}_{X,Y}\}$.
\end{lemma}
\begin{proof}
Suppose not.
By Lemma~\ref{lem:41},
$\{d_S\mid S\in \mathcal{S}_1\}=\{15,6\}$.
Let $S\in \mathcal{R}_{X,Y}$ with $d_S=6$.
By Lemma~\ref{lem:int}(ii), $6\mid c_{SS^t}^Rk$ for $R\in SS^t\setminus\{\Delta_X\}$.
Thus, $3\mid c_{SS^t}^R$, which contradicts Lemma~\ref{lem:int}(ii).
\end{proof}

\begin{lemma}\label{lem:beta}
Suppose that each element of $\mathcal{R}_Y=\{\Delta_Y,R,R'\}$ is symmetric
and $\Pi_x=\{C_1,C_2,C_3\}$ is the equitable partition of $(Y,\mathcal{R}_Y)$ as in Section~2 for $x\in X$.
We define
\[\mbox{$\{\beta_{ij}\}_{1\leq i,j\leq 3}$ and $\{\gamma_{ij}\}_{1\leq i,j\leq 3}$}\]
such that
$\beta_{ij}=|R(y)\cap C_j|$ with $y\in C_i$ and
$\gamma_{ij}:=|R'(y)\cap C_j|$ with $y\in C_i$.
Then we have the following:
\begin{enumerate}
\item For each $i$ we have $\sum_{j=1}^3\beta_{ij}=k$;
\item For all $i$, $j$ with $i\ne j$ we have $\beta_{ij}+\gamma_{ij}=|C_j|$;
\item For each $i$ we have $\beta_{ii}+\gamma_{ii}=|C_i|-1$;
\item For all $i$, $j$ we have $|C_i|\beta_{ij}=\beta_{ji}|C_j|$;
\item We have $\beta_{11}+\beta_{22}+\beta_{33}=k-1$.
\end{enumerate}
\end{lemma}

\begin{proof}
The first four statements can be proved by checking the definition of equitable partitions
and using a double-way counting for $(C_i\times C_j)\cap R$.

Let $\mathcal{A}$ be the adjacency algebra of $\mathcal{C}_Y$ and $W$
the subspace spanned by the characteristic vectors of the cells of $\Pi_x$.
Then $W$ is a left $\mathcal{A}$-module
corresponding to the algebra homomorphism defined by $A_R\mapsto (\beta_{ij})$, $A_{R'}\mapsto (\gamma_{ij})$.

We claim that $W$ affords the regular character.
Let $\chi$ be the character afforded by $W$, i.e., the value of the adjacency matrix of $R$ is equal to
$\sum_{i=1}^3\beta_{ii}$.
Note that the character afforded by $W$ is integral valued but not a sum of principal character.
Since $\mathrm{dim}(W)=3$, it follows that $\chi$ is the sum of irreducible characters of $\mathcal{A}$.
This implies that $\chi$ is the regular character of $\mathcal{A}$, and, hence,
the trace of the matrix $(\beta_{ij})$ is equal to $k-1$ by Lemma~\ref{lem:psd}(ii) with Lemma~\ref{lem:int}(ii).
\end{proof}
\begin{proposition}
If $(k,m)=(30,61)$,
then $1\in \{d_S\mid S\in \mathcal{R}_{X,Y}\}$.
\end{proposition}
\begin{proof}
Suppose not.

By Lemma~\ref{lem:bound}(ii),(iii), $|\mathcal{S}_1|\leq 3$ and $|\mathcal{S}_3|\leq 1$.
if $|\mathcal{S}_3|=1$, then $2k\leq d_{\mathcal{S}_3}<m=2k+1$, a contradiction.
Thus, $|\mathcal{S}_3|=0$.

Since $|\mathcal{S}_2|\leq 1$, it follows from Lemma~\ref{lem:ba} that
the following are only possible cases of $\{d_S\mid S\in \mathcal{R}_{X,Y}\}$:
\[\mbox{$\{30, 25,6\}$; $\{30,15,16\}$; $\{30,10,21\}$;$\{15,36,10\}$;$\{15,6,40\}$.}\]

The first three cases do not occur by Proposition~\ref{prop:lambda}
since each of them contains a pair of relatively prime numbers.

Note that $\{|C_i|\mid i=1,2,3\}=\{d_S\mid S\in \mathcal{R}_{X,Y}\}$
where $\Pi_x=\{C_1,C_2,C_3\}$ as in Lemma~\ref{lem:beta}.
Without loss of generality we may assume that
\[\mbox{$C_i=S_i(x)$ for $i=1,2,3$}.\]

From now on we shall use Lemma~\ref{lem:beta} many times without mentioning.

Suppose that
\[(|C_1|,|C_2|,|C_3|)=(10,15,36).\]
Since $|C_2|\beta_{23}=|C_3|\beta_{32}$,
we have $12\mid \beta_{23}$.
If $\beta_{23}\in \{0,36\}$, then $|S_2S_3^t|=1$, which
contradicts Proposition~\ref{prop:1}.
Replacing $R\in \mathcal{R}_Y$ by $R'$ if necessary we may assume that
$\beta_{23}=24$, and hence, $\beta_{32}=10$.

Since \[|C_1|\beta_{13}=|C_3|\beta_{31},\]
we have $18\mid \beta_{13}$.
If $\beta_{13}\in \{0,36\}$, then $|S_3S_1^t|=1$, which
contradicts Proposition~\ref{prop:lambda}.
Thus, $\beta_{13}=18$, and, hence, $\beta_{31}=5$.

By Lemma~\ref{lem:beta}(i),
\[\mbox{$\beta_{33}=15$, $\beta_{21}+\beta_{22}=6$ and
$\beta_{11}+\beta_{12}=12$.}\]
By Lemma~\ref{lem:beta}(v), $\beta_{11}+\beta_{22}=23$.
Thus, $\beta_{12}+\beta_{21}=19$, which contradicts $10\beta_{12}=15\beta_{21}$.
Therefore, $(d_S,d_T,d_U)=(10,15,36)$ does not occur.

Suppose \[(|C_1|,|C_2|,|C_3|)=(15,6,40).\]
Since $|C_2|\beta_{23}=|C_3|\beta_{32}$,
we have $20\mid \beta_{23}$.
If $\beta_{23}\in \{0,40\}$, then $|S_2S_3^t|=1$, which
contradicts Proposition~\ref{prop:lambda}.
We may assume that $\beta_{23}=20$, and hence, $\beta_{32}=3$.

Since
\[|C_1|\beta_{12}=|C_2|\beta_{21},\]
we have $5\mid \beta_{21}$.
By Lemma~\ref{lem:beta}(i),
\[\beta_{21}+\beta_{22}+20=30.\]
Thus, $5\mid \beta_{22}$.
Replacing $R\in \mathcal{R}_Y$ by $R'$ if necessary we may assume that
$\beta_{22}=5$, and hence, $\beta_{21}=5$ and $\beta_{12}=2$.

By Lemma~\ref{lem:beta}(i),(v),
we have
\[\mbox{$\beta_{11}+\beta_{13}=28$, $\beta_{31}+\beta_{33}=27$ and $\beta_{11}+\beta_{33}=24$.}\]
Thus, $\beta_{13}+\beta_{31}=31$,
which contradicts $15\beta_{13}=40\beta_{31}$.

This completes the proof.
\end{proof}

\begin{proposition}
If $(k,m)=(44,89)$,
then $1\in \{d_S\mid S\in \mathcal{R}_{X,Y}\}$.
\end{proposition}
\begin{proof}
Suppose not. By Lemma~\ref{lem:41},
the following is a unique possible case of $\{d_S\mid S\in \mathcal{R}_{X,Y}\}$:
\[\{12,33,44\}.\]

Without loss of generality we may assume that
\[\mbox{$C_i=S_i(y)$ for $i=1,2,3$ and $(|C_1|,|C_2|,|C_3|)=(12,33,44)$.}\]
Since $12\beta_{12}=33\beta_{21}$, $\beta_{12}\in \{0,11,22,33\}$.
Proposition~\ref{prop:lambda} forces $\beta_{12}\in \{11,22\}$,
and we may assume that $\beta_{12}=22$ by replacing $R\in \mathcal{R}_Y$ by $R'$.
Then $\beta_{21}=8$.

Note that $11$ divides $\beta_{13}$ and so does $\beta_{11}$ by Lemma~\ref{lem:beta}(i).
We divide our consideration into the following two cases $\beta_{11}=11$ or $0$.

Suppose $\beta_{11}=11$. Then $\beta_{13}=11$ and $\beta_{31}=3$.
By Lemma~\ref{lem:beta}(i),(v),
\[\mbox{$\beta_{22}+\beta_{23}=36$, $\beta_{32}+\beta_{33}=41$ and $\beta_{22}+\beta_{33}=32$.}\]
Therefore, $\beta_{23}+\beta_{32}=45$, which contradicts $33\beta_{23}=44\beta_{32}$.

Suppose $\beta_{11}=0$. Then
\[\mbox{$\beta_{13}=22$ and $\beta_{31}=6$.}\]
By Lemma~\ref{lem:beta}(i),(v),
\[\mbox{$\beta_{22}+\beta_{23}=36$, $\beta_{32}+\beta_{33}=38$ and $\beta_{22}+\beta_{33}=43$.}\]
Therefore, $\beta_{23}+\beta_{32}=34$, which contradicts $33\beta_{23}=44\beta_{32}$.

This completes the proof.
\end{proof}

\begin{lemma}\label{lem:71}
If $(k,m)=(35,71)$
then $\{d_S\mid S\in \mathcal{R}_{X,Y}\}=\{15,21,35\}$.
\end{lemma}
\begin{proof}
Applying Lemma~\ref{lem:bound}(i),(iii) and Lemma~\ref{lem:ba}
we conclude that $\{15,21,35\}$ is a unique case of
$\{d_S\mid S\in \mathcal{R}_{X,Y}\}$.
\end{proof}

We notice that the lemmata given in this section justify
the elimination given in Introduction.

\vspace{10mm}

\textbf{Acknowledgement}
First, the authors would like to express gratitude to Professor Akihiro Munemasa for his valuable suggestions
in revising the Theorem \ref{thm:main2} of the paper.

Second, the authors are grateful to anonymous referees for their kind attention to our work, care-
fully checking and error corrections.
The introduction was revised based on a comment of one of the referees, we appreciate that.

The first author was supported by the Persian Gulf University Research Council under the contract No: \textbf{PGU/FS/49--1/1390/1468}.

\end{document}